\documentclass[bj]{imsart}

\RequirePackage{amsthm,amsmath,amsfonts,amssymb}
\RequirePackage[numbers]{natbib}

\usepackage[utf8]{inputenc}
\usepackage{amsthm,amsmath}
\usepackage{amssymb}
\usepackage{graphicx}
\usepackage{color}
\usepackage{enumitem}
\usepackage[normalem]{ulem}

\RequirePackage[numbers]{natbib}
\RequirePackage[colorlinks,citecolor=blue,urlcolor=blue]{hyperref}

\startlocaldefs
\theoremstyle{plain}
\newtheorem{ther}{Theorem}[section]
\newtheorem{lem}[ther]{Lemma}

\newtheorem{prop}[ther]{Proposition}

\newtheorem{remark}[ther]{Remark}
\newtheorem{assu}[ther]{Assumption}
\newtheorem{exam}[ther]{Example}

\newcommand{\bea}{\begin{eqnarray}}
\newcommand{\eea}{\end{eqnarray}}
\newcommand{\ba}{\begin{eqnarray*}}
\newcommand{\ea}{\end{eqnarray*}}
\newcommand{\nn}{\nonumber}
\newcommand{\bq}{\begin{eqnarray}}
\newcommand{\eq}{\end{eqnarray}}

\endlocaldefs

\begin{document}

\begin{frontmatter}
\title{The coalescent tree of a Markov branching process with generalised logistic growth}

\begin{aug}
\author{\fnms{David} \snm{Cheek}}

\address{Massachusetts General Hospital, Harvard Medical School\\
149 13th St, Charlestown, MA 02129, USA\\dmcheek7@gmail.com}
\end{aug}

\begin{abstract}
We consider a class of density-dependent branching processes which generalises exponential, logistic and Gompertz growth. A population begins with a single individual, grows exponentially initially, and then growth may slow down as the population size moves towards a carrying capacity. At a time while the population is still growing superlinearly, a fixed number of individuals are sampled and their coalescent tree is drawn. Taking the sampling time and carrying capacity simultaneously to infinity, we prove convergence of the coalescent tree to a limiting tree which is in a sense universal over our class of models.
\end{abstract}
%

\end{frontmatter}

\section{Introduction}
A \textit{coalescent tree} is a random tree which models the genealogical structure of individuals sampled from a population, serving as a stepping stone between genetic data and models of population evolution. The most famous example is Kingman's coalescent \cite{Kingman,KingmanOrigins}. Kingman's coalescent has been incredibly helpful in the statistical analysis of genetic data \cite{Wakeley}, one reason being its robustness: it arises as the limiting coalescent tree from a wide class of models of population evolution, including the Wright-Fisher and Moran processes \cite{Mohle}. In this paper we look for a robustness result on the coalescent tree of a supercritical branching process.

Growing populations beginning with a single individual are found everywhere in biology. Examples are a virus spreading through a population, an animal colonising a new habitat, or a growing population of cancer cells. Such populations are naturally and commonly modelled by supercritical branching processes. However there is a question mark on the realism of supercritical branching processes. Although exponential growth is often realistic while a population is small, exponential growth is never indefinitely realistic. Instead, due to competition for limited resources, populations see their growth slow down as they become larger.

For greater biological realism, density-dependent branching processes have been studied \cite{Klebaner,Hopfner,JagersKlebaner} where individuals reproduce and die at rates which depend on the population size. Branching processes with logistic growth, where the per-individual death rate is linear with the population size, are especially popular \cite{Lambertlog,Campbell}. The downside to density-dependence is of course greatly increased complexity of analysis. Coalescent trees from density-dependent branching processes are apparently unknown.

In this paper we consider a class of density-dependent branching processes which includes exponential, logistic and Gompertz growth models as special cases. Our focus is the continuous-time Markov setting. We refer to the class of models as a `branching process with generalised logistic growth' or a `generalised logistic branching process'. The assumptions defining this class are minimal, namely that the population begins with one individual, that the offspring distribution has a uniformly bounded second moment, and that growth is exponential while the population size is small relative to some carrying capacity. Growth may be progressively constrained as the population size moves towards its carrying capacity. At a time while growth is still superlinear, a fixed number of individuals are sampled from the population and their coalescent tree is drawn. We consider the limit as the sampling time and carrying capacity simultaneously tend to infinity. We prove that the coalescent tree of a generalised logistic branching process coincides, in the limit, with the coalescent tree of a density-\textit{independent} supercritical branching process, whose structure was recently determined \cite{Johnston}.

Our result can be viewed as a universality or robustness result for the coalescent tree of a supercritical branching process, where the robustness is with respect to changes in the reproduction law while the population is large. However the tree crucially depends upon the reproduction law while the population is small, whose parameter space is infinite-dimensional. Thus the `supercritical coalescent' is necessarily robust in a weaker sense than Kingman's coalescent.



The paper is organised as follows. In Section \ref{modelsection}, we introduce the class of density-dependent branching processes and present examples. In Section \ref{mainresultsection}, we state convergence of the coalescent tree and describe its limiting form with reference to recent branching process coalescence works. In Sections \ref{expcoupsection} through to \ref{prooffinishsection}, we prove convergence of the coalescent tree. In Section \ref{expcoupsection}, we couple the population genealogies of a generalised logistic branching process and a supercritical branching process. In Section \ref{geneticdriftsection}, we show that genetic drift is negligible while the population is large and growing superlinearly, and hence the coupled population genealogies asymptotically coincide with respect to a certain representation. In Section \ref{samplingsection}, we translate the viewpoint from the population level to the sample level. In Section \ref{prooffinishsection}, we show convergence of the sampled individuals' ancestral identities and coalescence times.

\section{A generalised logistic branching process}\label{modelsection}
\subsection{Reproduction rates}
The population comprises an integer number of individuals. The individuals reproduce and die at rates which depend on the population size and on some index $\kappa\in\mathbb{N}$, whose purpose is that taking $\kappa\rightarrow\infty$ will represent the large carrying capacity limit. For a given $\kappa\in\mathbb{N}$ and while the population size is $n\in\mathbb{N}$, each individual in the population is replaced by $i\in\mathbb{N}\cup\{0\}$ individuals at rate $\Lambda^{\kappa,n}_i\geq0$. We shall soon describe the population evolution in more detailed language. First we state assumptions on the $\Lambda^{\kappa,n}_i$.
\begin{assu}\label{boundedsecondmoment} The reproduction rate distribution has a uniformly bounded second moment: $\sup_{\kappa,n}\sum_ii^2\Lambda^{\kappa,n}_i<\infty$.
\end{assu}
\begin{assu}\label{asslcc}
There are `intrinsic' reproduction rates which give positive growth. That is, there are $\Lambda^\infty_i\geq0$ for $i\in\mathbb{N}\cup\{0\}$ with  $$\sum_i(i-1)\Lambda^\infty_i>0,$$ and there exists $(n_\kappa)_{\kappa\in\mathbb{N}}\in\mathbb{N}^\mathbb{N}$ with $$\lim_{\kappa\rightarrow\infty}n_\kappa=\infty\quad\text{and}\quad\lim_{\kappa\rightarrow\infty}\sup_{n\leq n_\kappa}\sum_{i=0}^\infty|\Lambda_i^{\kappa,n}-\Lambda^\infty_i|=0.$$
\end{assu}
Assumption \ref{asslcc} defines the sense in which $\kappa\rightarrow\infty$ is the large carrying capacity limit, which will be illuminated by examples in Section \ref{examsec}.
\subsection{Genealogical structure}
 Following standard notation, individuals in the population are denoted by elements of $$\mathcal{T}=\cup_{m\in\mathbb{N}\cup\{0\}}\mathbb{N}^m.$$The set $\mathcal{T}$ has a partial ordering $\prec$ defined by $(u_l)_{l=1}^m\prec (v_l)_{l=1}^n$ if and only if $m<n$ and $(u_l)_{l=1}^m= (v_l)_{l=1}^m$. We write $\preceq$ for $\prec$ or $=$. The possible daughters of $u=(u_1,..,u_m)\in\mathcal{T}$ are denoted $ui=(u_1,..,u_m,i)$ for $i\in\mathbb{N}$.

For each $\kappa\in\mathbb{N}$ the population evolution is described by a continuous-time Markov process,
$(\mathcal{N}_\kappa(t))_{t\geq0}$, whose state space is the finite subsets of $\mathcal{T}$. The process has initial condition
$$
\mathcal{N}_\kappa(0)=\{\emptyset\}=\mathbb{N}^0
$$and sees transitions
$$
\mathcal{N}_\kappa(t)\mapsto\mathcal{N}_\kappa(t)\cup\left(\bigcup_{i=1}^j\{ui\}\right)\backslash\{u\}\quad\text{at rate }\Lambda_j^{\kappa,\#\mathcal{N}_\kappa(t)}I\{u\in\mathcal{N}_\kappa(t)\}
$$
 for $j\in\mathbb{N}\cup\{0\}$, where $\bigcup_{i=1}^0\{ui\}=\emptyset$ and $I$ is the indicator function. We shall write
$$
N_\kappa(t)=\#\mathcal{N}_\kappa(t)
$$
for the population size at time $t\geq0$. Assumption \ref{boundedsecondmoment} ensures that $N_\kappa(t)$ does not explode in finite time, so $\mathcal{N}_\kappa(t)$ is always well defined.

\subsection{Sampling time}
The sampling time $T_\kappa$ is a positive random variable which satisfies the following.
\begin{assu}The sampling time is a stopping time with respect to the filtration $(\mathcal{F}_\kappa(t))_{t\geq0}$ where $\mathcal{F}_\kappa(t)$ is the sigma-algebra generated by $(\mathcal{N}_\kappa(s))_{s\in[0,t]}$.
\end{assu}
\begin{assu}The sampling time is large: $\lim_{\kappa\rightarrow\infty}T_\kappa=\infty$ in distribution.
\end{assu}
\begin{assu}\label{sclb}Growth is superlinear before the sampling time: there exist $\alpha\in(0,1)$ and $c>0$ such that
$$
\lim_{\kappa\rightarrow\infty}\mathbb{P}\left[\forall t\in[0,T_\kappa],\bar{\Lambda}^{\kappa,N_\kappa(t)}\geq \frac{c}{N_\kappa(t)^{\alpha}}\right]=1,
$$
where $\bar{\Lambda}^{\kappa,N_\kappa(t)}=\sum_{i=0}^\infty(i-1)\Lambda_i^{\kappa,N_\kappa(t)}$ is the per-individual growth rate.
\end{assu}Assumption \ref{sclb} can be made sense of by reference to a deterministic growth model. Consider a population of size $n(t)\geq0$ at time $t\geq0$, with $n(0)=1$ and $\frac{dn}{dt}\geq cn^{1-\alpha}$. Then $n(t)\geq (1+c\alpha t)^{1/\alpha}$, which is superlinear.
\subsection{Examples}\label{examsec}
\begin{exam}\label{supex}Exponential growth: $\Lambda_i^{\kappa,n}=\Lambda_i^\infty,T_\kappa=\kappa$.
\end{exam}
\begin{exam}\label{logex1}Logistic growth with death due to competition: $$\Lambda_i^{\kappa,n}=\begin{cases}\Lambda_0^\infty+\frac{n}{\kappa}\sum_{j=0}^\infty  (j-1)\Lambda^\infty_j,\quad& i=0,\\\Lambda_i^\infty,\quad& i\geq1,\end{cases}$$ and $T_\kappa=\min\{t\geq0:N_\kappa(t)\geq x\kappa\}$ for some $x\in(0,1)$. Note the corresponding deterministic growth model: $\frac{dn}{dt}=\bar{\Lambda}^\infty n\left(1-\frac{n}{\kappa}\right)$, where $\bar{\Lambda}^\infty=\sum_{j=0}^\infty(j-1)\Lambda^\infty_j$ is the per-individual growth rate in the absence of a carrying capacity.
\end{exam}
\begin{exam}\label{gompex}
Gompertz growth with death due to competition:$$\Lambda_i^{\kappa,n}=\begin{cases}\Lambda^\infty_0+\frac{\log(n)}{\log(\kappa)}\sum_{j=0}^\infty  (j-1)\Lambda^\infty_j,\quad&i=0,\\\Lambda^\infty_i,\quad&i\geq1,\end{cases}$$ and $T_\kappa=\min\{t\geq0:\log(N_\kappa(t))\geq x\log(\kappa)\}$ for some $x\in(0,1)$. Note the corresponding deterministic growth model: $\frac{dn}{dt}=\bar{\Lambda}^\infty n\left(1-\frac{\log(n)}{\log(\kappa)}\right)$.
\end{exam}
\begin{exam}\label{logex2}Logistic growth with reduced reproduction due to competition:$$\begin{cases}\Lambda_0^{\kappa,n}=\Lambda^\infty_0,\\\Lambda_2^{\kappa,n}=\Lambda^\infty_2+\frac{n}{\kappa}\left(-\Lambda^\infty_2+\Lambda^\infty_0\right),\\\Lambda_i^{\kappa,n}=\Lambda^\infty_i=0,\quad&i\not\in\{0,2\},\end{cases}$$
 and $T_\kappa=\min\{t\geq0:N_\kappa(t)\geq x\kappa\}$ for some $x\in(0,1)$.
\end{exam}

\begin{exam}Exponential-to-polynomial growth with reduced reproduction due to competition:
$$
\begin{cases}\Lambda_0^{\kappa,n}=\Lambda^\infty_0,\\
\Lambda_2^{\kappa,n}=\Lambda^\infty_0+\frac{\kappa}{\kappa+n^{a}}(\Lambda^\infty_2-\Lambda^\infty_0),\\
\Lambda_i^{\kappa,n}=\Lambda^\infty_i=0,\quad &i\not\in\{0,2\},\end{cases}
$$
for some $a\in[0,1)$. Here $T_\kappa$ may be anything which converges to infinity.
\end{exam}

\begin{remark}For Examples \ref{logex1}, \ref{gompex} and \ref{logex2}, $\kappa$ is the population's carrying capacity in the sense that the population has zero growth rate while the population size is equal to $\kappa$.\end{remark}

\section{Coalescent tree}\label{mainresultsection}
Let $m\in\mathbb{N}$. On the event $\{N_\kappa(T_\kappa)\geq m\}$ that at least $m$ individuals are alive at time $T_\kappa$, let $U^1_\kappa,..,U^m_\kappa$ be a uniform sample without replacement of $m$ individuals from the population $\mathcal{N}_\kappa(T_\kappa)$. For $t\geq0$ and $\kappa\in\mathbb{N}$, let $\sim_{t,\kappa}$ be the equivalence relation on $\{1,..,m\}$ defined by
$$
i\sim_{t,\kappa}j\iff \exists u\in\mathcal{N}_\kappa(t), u\preceq U_\kappa^i\text{ and }u\preceq U_\kappa^j.
$$Let $\pi_\kappa(t)$ be the partition of $\{1,..,m\}$ induced by $\sim_{t,\kappa}$. That is, $i,j\in\{1,..,m\}$ are in the same block of $\pi_\kappa(t)$ if and only if $U_\kappa^i$ and $U^j_\kappa$ share an ancestor at time $t$. The \textit{coalescent tree} is the partition-valued process $(\pi_\kappa(t))_{t\geq0}$. Representing coalescent trees as partition-valued processes is standard \cite{Berestycki}. Less standard is that we view the tree forwards in time, with elements of the partition breaking up. This contrasts with the traditional view of coalescent trees backwards in time, where elements of the partition eponymously coalesce.

\subsection{Exponential growth}
Before stating our main result which is convergence of the coalescent tree, we describe the limiting coalescent tree with reference to recent results for density-\textit{independent} branching processes.

The coalescent trees of branching processes have been extensively studied \cite{Popovic,LambertStadler,LambertPopovic,BertoinLeGall,Hong,AthreyaHong,Athreya,HJR,Lambertbinary,JLPoisson}. In particular, we note that Johnston \cite{Johnston} recently gave a broad account of the coalescent tree arising from a continuous-time Galton-Watson branching process, including a limit result for the supercritical regime - which is our interest here. The coalescent tree determined by Johnston requires a little more notation. Let $(N_\infty(t))_{t\geq0}$ be a continuous-time Markov process on $\mathbb{N}\cup\{0\}$ with  transitions $N_\infty(t)\mapsto N_\infty(t)+i-1$ at rate $N_\infty(t)\Lambda^\infty_i$ for $i\in\mathbb{N}\cup\{0\}$, and with initial condition $N_\infty(0)=1$. In words, $N_\infty(t)$ is the population size if the carrying capacity were infinite. The per-individual growth rate is denoted $\bar{\Lambda}^\infty=\sum_i(i-1)\Lambda_i^\infty$. Let $H_t^j(s)=\frac{d^j}{ds^j}\mathbb{E}s^{N_\infty(t)}$ for $s,t\geq0$ and $j\in\mathbb{N}$.  Let $\psi^j(v)=\frac{d^j}{dv^j}\mathbb{E}e^{-vW}$ for $v\geq0$ and $j\in\mathbb{N}\cup\{0\}$, where $W=\lim_{t\rightarrow\infty}e^{-\bar{\Lambda}^\infty t}N_\infty(t)$ (that the limit $W\in[0,\infty)$ exists almost surely is a classic branching process result \cite{an}). Finally let $(\pi_\infty(t))_{t\geq0}$ be a process on the partitions of $\{1,..,m\}$, distributed as
\begin{align}
&\mathbb{P}[\pi_\infty(t_1)=\gamma_1,..,\pi_\infty(t_r)=\gamma_r]\nn\\
&=\frac{(-1)^me^{-m\bar{\Lambda}^\infty t_r}}{1-\psi(\infty)}\int_0^\infty \frac{v^{m-1}}{(m-1)!}\prod_{i=0}^{r-1}\prod_{\Gamma\in\gamma_i}H_{t_i-t_{i-1}}^{b_i(\Gamma)}\left(\psi(e^{-\bar{\Lambda}^\infty t_{i+1}}v)\right)\nn\\
&\quad\quad\quad\quad\quad\quad\quad\quad\quad\quad\quad\quad\quad\quad\times\prod_{\Gamma\in\gamma_r}\psi^{\#\Gamma}\left(e^{-\bar{\Lambda}^\infty t_r}v\right)dv,\label{ctd}
\end{align}
where $0<t_1<...<t_r$; and $\gamma_1,..,\gamma_r$ are partitions of $\{1,..,m\}$ such that each $\Gamma\in\gamma_i$ is the union of $b_i(\Gamma)\geq1$ elements of $\gamma_{i+1}$, for $i=1,..,r-1$.
\begin{ther}[Theorem 3.5 of \cite{Johnston}]\label{Jcor}Consider Example \ref{supex}, which is a density-independent supercritical branching process. As $\kappa\rightarrow\infty$, $\pi_\kappa$ on the event $\{N_\kappa(T_\kappa)\geq m\}$ converges in finite dimensional distributions to $\pi_\infty$. 
\end{ther}
\begin{remark}The coalescent tree $\pi_\kappa$ is characterised by the coalescence times $$\tau_\kappa^{i,j}=\sup\{t\geq0:\exists u\in\mathcal{N}_\kappa(t),u\preceq U_\kappa^i\text{ and }u\preceq U_\kappa^j\},\quad 1\leq i<j\leq m,$$ and convergence of the coalescent tree is equivalent to convergence of the coalescence times.
\end{remark}
While the distribution of the coalescent tree (\ref{ctd}) may seem tricky at first glance, matters simplify considerably in the case of a binary branching process: $\Lambda_i^{\infty}=0$ for $i\not\in\{0,2\}$. Here the coalescent tree is binary too and there are exactly $m-1$ distinct coalescence times. We denote the unordered coalescence times, in the limit, as $\sigma_\infty^1,..,\sigma_\infty^{m-1}$. The distribution of the $\sigma_\infty^i$ can be obtained from (\ref{ctd}), or more quickly from Harris, Johnston and Roberts' Theorem 1 \cite{HJR}.
Taking the large time limit of their Theorem 1 in the supercritical case gives that the coalescence times are distributed as
\bq\label{ctimes}
&&\mathbb{P}[\sigma_\infty^1\geq t_1,..,\sigma_\infty^{m-1}\geq t_{m-1}]\\&&=m\left[\prod_{i=1}^{m-1}\frac{-e^{-\bar{\Lambda}^\infty t_i}}{1-e^{-\bar{\Lambda}^\infty t_i}}+\sum_{j=1}^{m-1}\frac{\bar{\Lambda}^\infty t_je^{-\bar{\Lambda}^\infty t_j}}{(1-e^{-\bar{\Lambda}^\infty t_j})^2}\left(\prod_{\substack{i=1\\i\not=j}}^{m-1}\frac{e^{-\bar{\Lambda}^\infty t_i}}{e^{-\bar{\Lambda}^\infty t_i}-e^{-\bar{\Lambda}^\infty t_j}}\right)\right]\nn
\eq
for  $t_1,..,t_{m-1}>0$, and the coalescence times are independent of the tree topology which is that of Kingman's coalescent (viewed backwards in time, each coalescence event involves a uniformly chosen pair of partition elements).

Lambert \cite{Lambertbinary} also determined the coalescent tree for a density-independent binary branching process. He elegantly described the coalescent tree in terms of a mixture of coalescent point processes, giving an effectively instantaneous method of tree simulation. Taking the large time limit of his result in the supercritical case recovers (\ref{ctimes}).

\subsection{Generalised logistic growth}

The paper's main result generalises Theorem \ref{Jcor} to a branching process with generalised logistic growth.
\begin{ther}\label{mainresult}
As $\kappa\rightarrow\infty$, $\pi_\kappa$ on the event $\{N_\kappa(T_\kappa)\geq m\}$ converges in finite dimensional distributions to $\pi_\infty$, where $\pi_\infty$ is distributed as (\ref{ctd}). That is, the limiting coalescent tree is parameterised by the intrinsic reproduction rates.
\end{ther}
Theorem \ref{mainresult} serves two purposes. First, it demonstrates that the coalescent tree of a supercritical Markov branching process is robust with respect to deviations from exponential growth at large sizes, suggesting that the tree may be applied to genetic data from populations with a variety of growth laws. Second, it is to our knowledge a first complete description of the coalescent tree from a density-dependent branching process. Hopefully the result and proof will inspire more general explorations of coalescence in density-dependent branching processes in the future.

There are open questions. What if individuals are sampled after a carrying capacity is reached? What if the intrinsic offspring distribution is critical, subcritical, or heavy-tailed? What if reproduction and death rates depend on the ages of individuals? Regarding the third question we note that there is an extensive history of research on age-dependent, density-independent branching processes (see Chapter 4 of \cite{an} for an introduction) and there are partial results on their coalescent trees \cite{LambertStadler,Hong,AthreyaHong}. To expand coalescence theory to branching processes with both age-dependence and density-dependence (such as the model in \cite{JagersKlebaner}) would do a great service for the aim of biological realism.

The remainder of the paper is concerned with Theorem \ref{mainresult}'s proof.
\section{Coupling with exponential growth}\label{expcoupsection}
At the base of the proof is a coupling between generalised logistic growth and exponential growth. Exponential growth is described by  $(\mathcal{N}_\infty(t))_{t\geq0}$, a continuous-time Markov process whose state space is the finite subsets of $\mathcal{T}$. The process sees transitions
$$
\mathcal{N}_\infty(t)\mapsto\mathcal{N}_\infty(t)\cup\left(\bigcup_{i=1}^j\{ui\}\right)\backslash\{u\}
$$
 at rate $\Lambda^\infty_jI\{u\in\mathcal{N}_\infty(t)\}$ for $j\in\mathbb{N}\cup\{0\}$ and $u\in\mathcal{T}$. The initial condition is
$$
\mathcal{N}_\infty(0)=\{\emptyset\}=\mathbb{N}^0.$$
Let
$$
N_\infty(t)=\#\mathcal{N}_\infty(t)
$$
be the exponentially growing population size.
Now put the $\mathcal{N}_\kappa(\cdot)$ for $\kappa\in\mathbb{N}\cup\{\infty\}$ on a single probability space. We define their joint distribution by the following.
\begin{enumerate}
\item  The $\mathcal{N}_\kappa(\cdot)$, $\kappa\in\mathbb{N}$, are conditionally independent given $\mathcal{N}_\infty(\cdot)$.
    \item
$\begin{pmatrix}\mathcal{N}_\infty(t)\\\mathcal{N}_\kappa(t)\end{pmatrix}$ transitions to
$$
\begin{cases}\begin{pmatrix}\mathcal{N}_\infty(t)\cup\left(\bigcup_{i=1}^j\{ui\}\right)\backslash\{u\}\\\mathcal{N}_\kappa(t)\cup\left(\bigcup_{i=1}^j\{ui\}\right)\backslash\{u\}\end{pmatrix}\\\quad\quad\quad\text{at rate }\min\{\Lambda^\infty_j,\Lambda_j^{\kappa,N_\kappa(t)}\} I\{u\in \mathcal{N}_\infty(t)\cap \mathcal{N}_\kappa(t)\},\\
&~\\
\begin{pmatrix}\mathcal{N}_\infty(t)\cup\left(\bigcup_{i=1}^j\{ui\}\right)\backslash\{u\}\\\mathcal{N}_\kappa(t)\end{pmatrix}\\
\quad\quad\quad\text{at rate }(\Lambda^\infty_j-\Lambda_j^{\kappa,N_\kappa(t)}) I\{\Lambda^\infty_j>\Lambda_j^{\kappa,N_\kappa(t)},u\in \mathcal{N}_\infty(t)\cap \mathcal{N}_\kappa(t)\}\\\quad\quad\quad\quad\quad\quad\quad+\Lambda^\infty_j I\{u\in \mathcal{N}_\infty(t)\backslash \mathcal{N}_\kappa(t)\},\\
&~\\
\begin{pmatrix}\mathcal{N}_\infty(t)\\\mathcal{N}_\kappa(t)\cup\left(\bigcup_{i=1}^j\{ui\}\right)\backslash\{u\}\end{pmatrix}\\\quad\quad\quad\text{at rate }(\Lambda_j^{\kappa,N_\kappa(t)}-\Lambda^\infty_j) I\{\Lambda_j^{\kappa,N_\kappa(t)}>\Lambda^\infty_j,u\in \mathcal{N}_\infty(t)\cap \mathcal{N}_\kappa(t)\}\\\quad\quad\quad\quad\quad\quad\quad+\Lambda_j^{\kappa,N_\kappa(t)}I\{u\in \mathcal{N}_\kappa(t)\backslash \mathcal{N}_\infty(t)\},\end{cases}
$$
for $j\in\mathbb{N}\cup\{0\}$ and $u\in\mathcal{T}$.
\end{enumerate}
The conditional independence of the $\mathcal{N}_\kappa(\cdot)$ will not be important. We only specify the conditional independence so that \textit{some} joint distribution is specified. On the other hand, the joint transition rates of $\mathcal{N}_\kappa(\cdot)$ and $\mathcal{N}_\infty(\cdot)$ will be helpful for the proof. This coupling maximises the time for which $\mathcal{N}_\kappa(\cdot)$ and $\mathcal{N}_\infty(\cdot)$ remain equal. Initially $\mathcal{N}_\kappa(\cdot)$ and $\mathcal{N}_\infty(\cdot)$ are equal, but their state of equality is broken at rate $$N_\infty(t)\sum_{i=0}^\infty|\Lambda_i^\infty-\Lambda_i^{\kappa,N_\infty(t)}|.$$ Note that
\begin{eqnarray}
&&\mathbb{P}\left[\mathcal{N}_\infty(s)=\mathcal{N}_\kappa(s),\forall s\in[0,t]\right]\nonumber\\
&&=\mathbb{E}\left[\exp\left(-\int_0^tN_\infty(s)\sum_{i=0}^\infty|\Lambda^\infty_i-\Lambda_i^{\kappa,N_\infty(s)}|ds\right)\right].\label{equaltime}
\end{eqnarray}
The next result considers (\ref{equaltime}) in the view of the large carrying capacity limit. It says that the exponential and generalised logistic growth models coincide for a long initial time period.
\begin{lem}\label{coupeq}There exists a sequence of times $(\zeta_\kappa)_{\kappa\in\mathbb{N}}\in(0,\infty)^\mathbb{N}$ with 
$$\lim_{\kappa\rightarrow\infty}\zeta_\kappa=\infty\quad\text{and}\quad\lim_{\kappa\rightarrow\infty}\mathbb{P}\left[\mathcal{N}_\infty(s)=\mathcal{N}_\kappa(s),\forall s\in[0,\zeta_\kappa]\right]=1.$$
\end{lem}

\begin{proof}[Proof of Lemma \ref{coupeq}] Recall from Assumption \ref{asslcc} that $(n_\kappa)$ is some sequence of integers. Let
$$\delta_\kappa=\sup_{n\leq n_\kappa}\sum_{i=0}^\infty|\Lambda_i^{\kappa,n}-\Lambda^\infty_i|.
$$
Then on the event $\left\{\sup_{s\in[0,t]}N_\infty(s)\leq \min\{\delta_\kappa^{-1/2},n_\kappa\}\right\}$, we see the bound
$$
\int_0^{t}N_\infty(s)\sum_{i=0}^\infty|\Lambda^\infty_j-\Lambda_j^{\kappa,N_\infty(s)}|ds\leq t\delta_\kappa^{1/2}.
$$
It follows that we can bound (\ref{equaltime}) below by
\begin{eqnarray}
&&\mathbb{P}\left[\mathcal{N}_\infty(s)=\mathcal{N}_\kappa(s),\forall s\in[0,t]\right]\nonumber\\
&&=\mathbb{E}\left[\exp\left(-\int_0^{t}N_\infty(s)\sum_{i=0}^\infty|\Lambda^\infty_j-\Lambda_j^{\kappa,N_\infty(s)}|ds\right)\right]\nonumber\\
&&\geq\mathbb{E}\left[\exp\left(-t\delta_\kappa^{1/2} \right)I\left\{\sup_{s\in[0,t]}N_\infty(s)\leq \min\big\{\delta_\kappa^{-1/2},n_\kappa\big\}\right\}\right]\nonumber\\
&&=\exp\left(-t\delta_\kappa^{1/2} \right)\mathbb{P}\left[\sup_{s\in[0,t]}N_\infty(s)\leq \min\big\{\delta_\kappa^{-1/2},n_\kappa\big\}\right]\nonumber\\
&&\geq \exp\left(-t\delta_\kappa^{1/2} \right)\left[1-\frac{e^{\bar{\Lambda}^\infty t}}{\min\big\{\delta_\kappa^{-1/2},n_\kappa\big\}}\right],\label{lbep}
\end{eqnarray}
where the last inequality is Doob's martingale inequality ($e^{-\bar{\Lambda}^\infty t}N_\infty(t)$ is a martingale).
But Assumption \ref{asslcc} says that $\delta_\kappa$ converges to zero while $n_\kappa$ converges to infinity. Thus by replacing $t$ with
$$
\zeta_\kappa=\frac{\log\left(\min\{\delta_\kappa^{-1/2},n_\kappa\}\right)}{2\bar{\Lambda}^\infty},
$$
the lower bound (\ref{lbep}) converges to $1$.
\end{proof}
The explicit definition of $\zeta_\kappa$ seen in Lemma \ref{coupeq}'s proof won't be used. Only the properties of  $\zeta_\kappa$ which are seen in the Lemma's statement will be important.

Now for each member of the population we look at the abundance of their descendants. Let
$$
D_{\kappa}^u(t)=\#\{v\in\mathcal{N}_\kappa(t):u\preceq v\}
$$
be the number of individuals alive at time $t\geq0$ which are descendants of $u\in\mathcal{T}$ for $\kappa\in\mathbb{N}\cup\{\infty\}$. Let
$$
F_{\kappa}^u(t)=I\{N_\kappa(t)>0\}\frac{D_\kappa^u(t)}{N_\kappa(t)}
$$
be the fraction of the population at time $t$ which descended from $u$. We will refer to the $F_\kappa^u(t)$ as the descendant fractions.  The descendant fractions will play a crucial role in Theorem \ref{mainresult}'s proof because they will give the probabilities that sampled individuals have particular ancestors. The next result shows convergence of the descendant fractions in the exponential model.
\begin{lem}\label{expdf}For $u\in\mathcal{T}$, the limit
$$
\lim_{t\rightarrow\infty}F_\infty^u(t)=:F_\infty^u(\infty)\in[0,1]
$$
exists almost surely. 
\begin{proof}For any $j,\kappa,n\in\mathbb{N}$,
\begin{eqnarray*}
\sum_{i=1}^ji\log(i)\Lambda^\infty_i&\leq&j\log(j)\sum_{i=1}^j |\Lambda_i^\infty-\Lambda_i^{\kappa,n}|+ \sum_{i=1}^\infty i^2\Lambda_i^{\kappa,n}.
\end{eqnarray*}So by Assumption \ref{asslcc},
\begin{eqnarray*}
\sum_{i=1}^ji\log(i)\Lambda^\infty_i&\leq&\limsup_{\kappa\rightarrow\infty}j\log(j)\sum_{i=1}^j |\Lambda_i^\infty-\Lambda_i^{\kappa,n}|+ \sup_{\kappa,n}\sum_{i=1}^\infty i^2\Lambda_i^{\kappa,n}\\
&=& \sup_{\kappa,n}\sum_{i=1}^\infty i^2\Lambda_i^{\kappa,n}.
\end{eqnarray*}
Then by Assumption \ref{boundedsecondmoment}, $\sum_{i=1}^\infty i\log(i)\Lambda^\infty_i$ is finite, which gives the Kesten-Stigum condition for the supercritical branching process $N_\infty(\cdot)$ (see for example \cite{an} Chapter 3.7, Theorem 2). Therefore, recalling that $\bar{\Lambda}^\infty =\sum_i(i-1)\Lambda^\infty_i$ denotes the exponential growth rate,
\begin{equation}\label{classical}
\lim_{t\rightarrow\infty}e^{-\bar{\Lambda}^\infty t}N_\infty(t)=W\in[0,\infty)
\end{equation}
exists almost surely with $$\{W>0\}=\{N_\infty(t)>0,\forall t\}.$$ Let $\tau=\inf\{t\geq0:u\in\mathcal{N}_\infty(t)\}$ be the first time at which $u$ is alive and observe that, because $\mathcal{N}_\infty(\cdot)$ is a branching process, the number of descendants of $u$ follows the same asymptotic behaviour as (\ref{classical}):
\begin{equation}\label{classicaltr}
I\{\tau<\infty\}\lim_{t\rightarrow\infty}e^{-\bar{\Lambda}^\infty (t-\tau)}D_\infty^u(t)=I\{\tau<\infty\}W'
\end{equation}almost surely, where $W'\overset{d}{=}W$. Combining (\ref{classical}) and (\ref{classicaltr}),
\begin{eqnarray*}
\lim_{t\rightarrow\infty}F_\infty^u(t)&=&\lim_{t\rightarrow\infty}I\{\tau<\infty,N_\infty(t)>0\}e^{-\bar{\Lambda}^\infty \tau}\frac{e^{-\bar{\Lambda}^\infty (t-\tau)}D_\infty^u(t)}{e^{-\bar{\Lambda}^\infty t}N_\infty(t)}\\
&=&I\{\tau<\infty,W>0\}e^{-\bar{\Lambda}^\infty \tau}\frac{W'}{W}.
\end{eqnarray*}
\end{proof}
\end{lem}
Our final result of this section says that the descendant fractions viewed late during the exponential growth phase of the generalised logistic model converge, unsurprisingly, to the same limit as the descendant fractions in the exponential model.
\begin{lem}\label{expas}For $u\in\mathcal{T}$,
\ba
\lim_{\kappa\rightarrow\infty}F_\kappa^u(\min\{\zeta_\kappa,T_\kappa\})=F_\infty^u(\infty)
\ea
in probability, where $F_\infty^u(\infty)$ is defined in Lemma \ref{expdf}.
\begin{proof}Combine Lemmas \ref{coupeq} and \ref{expdf}.
\end{proof}
\end{lem}
This section has viewed the population's initial exponential growth phase. The next section will look at later times when the population growth may be constrained.


\section{Genetic drift}\label{geneticdriftsection}
The quantity $F_\kappa^u(T_\kappa)-F_\kappa^u(\min\{\zeta_\kappa,T_\kappa\})$ can be thought of as `genetic drift' on the time interval $[\zeta_\kappa,T_\kappa]$. In this section we will show that genetic drift is negligible.
Towards this end, the next result bounds genetic drift by the integrated reciprocal of the population size. As an aside note that the Moran and Wright-Fisher processes see analogous behaviour, with genetic drift occurring at a rate proportional to the reciprocal of the population size. 
\begin{lem}\label{gdib}Let $\tilde{T}_\kappa$ be a stopping time with respect to $(\mathcal{F}_\kappa(t))_{t\geq0}$ $($where $\mathcal{F}_\kappa(t)$ is the sigma-algebra generated by $(\mathcal{N}_\kappa(s))_{s\in[0,t]}$$)$. There exists $a>0$ such that for all $\kappa\in\mathbb{N}$ and $s\geq0$,
\ba
&&\mathbb{E}\left[(F^u_\kappa(\tilde{T}_\kappa)-F_\kappa^u(\min\{s,\tilde{T}_\kappa\})^2\right]\\&&\leq a\int_{s}^\infty\mathbb{E}\left[\frac{I\{N_\kappa(t)>0,t<\tilde{T}_\kappa\}}{N_\kappa(t)}\right]dt.
\ea
\end{lem}
\begin{proof}We look at the process
$$
G(t)=F_\kappa^u(\min\{t,\tilde{T}_\kappa\}),\quad t\geq0,
$$
which sees the following transitions placed into three categories. 
\begin{enumerate}\item  The mother $u^-=(u_i)_{i=1}^{l-1}$ of $u=(u_i)_{i=1}^l$ reproduces: for $n,i>0$,
\end{enumerate}
\ba
&&\lim_{h\downarrow0}h^{-1}\mathbb{P}\left[G(t+h)=\frac{1}{n+i-1}\Bigg|N_\kappa(t)=n,t<\tilde{T}_\kappa,u^-\in\mathcal{N}_\kappa(t)\right]\\&&=\sum_{i=u_l}^\infty\Lambda_i^{\kappa,n}.
\ea
\begin{enumerate}[resume]
\item A descendant of $u$ (including $u$) reproduces or dies: for integers  $n\geq d>0$ and $i\geq0$,
\end{enumerate}
\ba
&&\lim_{h\downarrow0}h^{-1}\mathbb{P}\left[G(t+h)=\frac{d+i-1}{n+i-1}\Bigg|N_\kappa(t)=n,D_\kappa^u(t)=d,t<\tilde{T}_\kappa\right]\\&&=\Lambda_i^{\kappa,n}d
\ea
\quad\quad \quad (for $d=n=1,i=0$ in the above we use the convention $0/0=0$).
\begin{enumerate}[resume]
\item An individual who is not $u$'s mother nor a descendant of $u$ reproduces or dies: for integers  $n> d>0$ and $i\geq0$,
\end{enumerate}
\ba
&&\lim_{h\downarrow0}h^{-1}\mathbb{P}\left[G(t+h)=\frac{d}{n+i-1}\Bigg|N_\kappa(t)=n,D_\kappa^u(t)=d,t<\tilde{T}_\kappa\right]\\&&=\Lambda_i^{\kappa,n}(n-d).
\ea
The transition rates give that
\ba
\frac{d}{dt}\mathbb{E}[G(t)G(s)]=-\Lambda_0^{\kappa,1}\mathbb{E}[I\{N_\kappa(t)=D_\kappa^u(t)=1,t<\tilde{T}_\kappa\}G(s)]
\ea
for $t> s$,
and
%
%
%
\ba
&&\frac{d}{dt}\mathbb{E}[G(t)^2]\nn\\
&&=\sum_{i=u_l}^\infty\mathbb{E}\left[\frac{\Lambda_i^{\kappa,N_\kappa(t)}I\{u^-\in\mathcal{N}_\kappa(t),t<\tilde{T}_\kappa\}}{(N_\kappa(t)+i-1)^2}\right]\\
&&\quad-\Lambda_0^{\kappa,1}\mathbb{E}[I\{N_\kappa(t)=D_\kappa^u(t)=1,t<\tilde{T}_\kappa\}]\label{prob1in2}\\
&&\quad+\sum_{i=0}^\infty\mathbb{E}\Bigg[\frac{\Lambda_i^{\kappa,N_\kappa(t)}(i-1)^2F_\kappa^u(t)(1-F_\kappa^u(t)) N_\kappa(t)}{(N_\kappa(t)+i-1)^2}\nonumber\\&&\quad\quad\quad\quad\quad\quad\quad\times I\{N_\kappa(t)>D_{\kappa}^u(t)>0,t<\tilde{T}_\kappa\}\Bigg].
\ea
It follows that for $t\geq s$,
\ba
&&\frac{d}{dt}\mathbb{E}\left[(G(t)-G(s))^2\right]\\
&&\leq\sum_{i=u_l}^\infty\mathbb{E}\left[\frac{\Lambda_i^{\kappa,N_\kappa(t)}I\{u^-\in\mathcal{N}_\kappa(t),t<\tilde{T}_\kappa\}}{(N_\kappa(t)+i-1)^2}\right]\\
&&\quad+\Lambda_0^{\kappa,1}\mathbb{E}[I\{N_\kappa(t)=1,t<\tilde{T}_\kappa\}]\label{prob1in2}\\
&&\quad+\sum_{i=0}^\infty\mathbb{E}\Bigg[\frac{ \Lambda_i^{\kappa,N_\kappa(t)}(i-1)^2N_\kappa(t)}{(N_\kappa(t)+i-1)^2} I\{N_\kappa(t)>1,t<\tilde{T}_\kappa\}\Bigg]\label{moment5}\\
&&\leq a'\mathbb{E}\left[\frac{I\{N_\kappa(t)>0,t<\tilde{T}_\kappa\}}{N_\kappa(t)}\sum_{i=0}^\infty\Lambda_i^{\kappa,N_\kappa(t)}(i-1)^2\right]\\
&&\leq a\mathbb{E}\left[\frac{I\{N_\kappa(t)>0,t<\tilde{T}_\kappa\}}{N_\kappa(t)}\right],
\ea
where $a',a>0$ are constants independent of $t$ and $\kappa$ with $a$ coming from Assumption \ref{boundedsecondmoment}. Finally
\ba&&\mathbb{E}\left[\left(G(\infty)-G(s)\right)^2\right]\\&&=\int_{s}^\infty\frac{d}{dt}\mathbb{E}[G(t)-G(s))^2]dt\\
&&\leq a\int_{s}^\infty\mathbb{E}\left[\frac{I\{N_\kappa(t)>0,t<\tilde{T}_\kappa\}}{N_\kappa(t)}\right]dt,
\ea
which is the result.\end{proof}
For the next result, let $$M_\kappa=\left\{n\in\mathbb{N}:\bar{\Lambda}^{\kappa,n}\geq c n^{-\alpha}\right\}$$
be the set of population sizes for which growth is superlinear (where $\alpha\in(0,1)$ and $c>0$ are seen in Assumption \ref{sclb}), and let
\bq\label{slt}
S_\kappa=\min\left\{t\geq0:N_\kappa(t)\not\in M_\kappa\right\}
\eq
be the first time at which the bound for superlinear growth is broken or the population becomes extinct.
\begin{lem}\label{gdib2}There exists $b>0$ such that for all $\kappa\in\mathbb{N}$ and $t>0$,
$$
\mathbb{E}\left[\frac{I\{t<S_\kappa\}}{N_\kappa(t)}\right]\leq bt^{-1/\alpha}.
$$
\end{lem}
\begin{proof}
Observe that
\ba
&&I\{n\in M_\kappa\}\sum_i\frac{(i-1)\Lambda^{\kappa,n}_i}{n+i-1}\\
&&\geq n^{-1}I\{n\in M_\kappa\}\sum_i(i-1)\Lambda^{\kappa,n}_i-\left|\sum_i\Lambda^{\kappa,n}_i\frac{(i-1)}{n+i-1}-\sum_i\Lambda^{\kappa,n}_i\frac{(i-1)}{n}\right|\\
&&= n^{-1}I\{n\in M_\kappa\}\sum_i(i-1)\Lambda^{\kappa,n}_i-\sum_i\Lambda^{\kappa,n}_i\frac{(i-1)^2}{n(n+i-1)}\\
&&\geq cn^{-1-\alpha}-dn^{-2}
\ea
for all $\kappa$ and $n$, where $d>0$ is some constant which exists by Assumption \ref{boundedsecondmoment}.
So
\bq\label{poslhs}
\inf_\kappa\left(I\{n\in M_\kappa\}\sum_i\frac{(i-1)\Lambda^{\kappa,n}_i}{n+i-1}\right)\geq cn^{-1-\alpha}-dn^{-2}
\eq
The left hand side of (\ref{poslhs}) is positive for all $n$ while the right hand side approaches $cn^{-1-\alpha}$ for large $n$. Hence there exists $\beta>0$ such that
\bq\label{willbeusingd}
\inf_\kappa\left(I\{n\in M_\kappa\}\sum_i\frac{(i-1)\Lambda^{\kappa,n}_i}{n+i-1}\right)\geq \beta n^{-1-\alpha}.
\eq
Now
\bq
&&\frac{d}{dt}\mathbb{E}\left[\frac{I\{t<S_\kappa\}}{N_\kappa(t)}\right]\nn\\
&&=-\sum_{i=0}^\infty\mathbb{E}\Bigg[\Lambda_i^{N_\kappa(t),\kappa}I\{t<S_\kappa\}\nn\\&&\quad\quad\times\Bigg(\frac{I\{N_\kappa(t)+i-1\not\in M_\kappa\}(i-1)}{N_\kappa(t)+i-1}+I\{N_\kappa(t)+i-1\in M_\kappa\}\Bigg)\Bigg]\nn\\
&&\leq-\mathbb{E}\left[I\{t<S_\kappa\}\sum_{i=0}^\infty\frac{(i-1)\Lambda^{\kappa,N_\kappa(t)}_i}{N_\kappa(t)+i-1}\right]\nn\\
&&\leq-\beta\mathbb{E}\left[\frac{I\{t<S_\kappa\}}{N_\kappa(t)^{1+\alpha}}\right]\label{usingd},
\eq
where (\ref{usingd}) is due to (\ref{willbeusingd}). Next using Jensen's inequality, (\ref{usingd}) becomes
\bq\label{mart}
\frac{d}{dt}\mathbb{E}\left[\frac{I\{t<S_\kappa\}}{N_\kappa(t)}\right]\leq-\beta\mathbb{E}\left[\frac{I\{t<S_\kappa\}}{N_\kappa(t)}\right]^{1+\alpha}
\eq
which, combined with the initial condition
$$
\mathbb{E}\left[\frac{I\{0<S_\kappa\}}{N_\kappa(0)}\right]\leq1,
$$
implies that
$$
\mathbb{E}\left[\frac{I\{t<S_\kappa\}}{N_\kappa(t)}\right]\leq r(t)
$$
where
$\frac{d}{dt}r(t)=-\beta r(t)^{1+\alpha}$ and $r(0)=1$. But $r(t)=(1+\alpha \beta t)^{-1/\alpha}$ so we are done.
\end{proof}
Lemma \ref{gdib} says that genetic drift is bounded by the expected reciprocal of the population size, integrated over time. Lemma \ref{gdib2} says  that the expected reciprocal of the population size decays faster than the reciprocal of time provided that the growth rate is superlinear. Assumption \ref{sclb} says that the growth rate is indeed superlinear before the sampling time. The next result ties these strands together, saying that genetic drift during the time interval $[\zeta_\kappa,T_\kappa]$ is negligible.
\begin{lem}\label{gdrisn}For $u\in\mathcal{T}$,
$$
\lim_{\kappa\rightarrow\infty}\left(F_\kappa^u(T_\kappa)-F_\kappa^u(\min\{\zeta_\kappa,T_\kappa\})\right)=0
$$
in probability.
\begin{proof}Let $\tilde{T}_\kappa=\min\{S_\kappa,T_\kappa\}$. By Lemmas \ref{gdib} and \ref{gdib2},
\bq
&&\mathbb{E}\left[(F_\kappa(\tilde{T}_\kappa)-F_\kappa^u(\min\{\zeta_\kappa,\tilde{T}_\kappa\})^2\right]\nn\\&&\leq a\int_{\zeta_\kappa}^\infty\mathbb{E}\left[\frac{I\{N_\kappa(t)>0,t<\tilde{T}_\kappa\}}{N_\kappa(t)}\right]dt\nn\\
&&\leq a\int_{\zeta_\kappa}^\infty\mathbb{E}\left[\frac{I\{t<S_\kappa\}}{N_\kappa(t)}\right]dt\nn\\
&&\leq ab\int_{\zeta_\kappa}^\infty t^{-1/\alpha}dt\nn\\
&&\rightarrow0\label{ttctz}
\eq
as $\kappa\rightarrow\infty$. Then
\ba
&&|F_\kappa^u(T_\kappa)-F_\kappa^u(\min\{\zeta_\kappa,T_\kappa\})|\nn\\
&&=|F_\kappa^u(T_\kappa)-F_\kappa^u(\min\{\zeta_\kappa,T_\kappa\})|I\{T_\kappa> S_\kappa\}\nn\\
&&\quad +|F_\kappa^u(\tilde{T}_\kappa)-F_\kappa^u(\min\{\zeta_\kappa,\tilde{T}_\kappa\})|I\{T_\kappa\leq S_\kappa\}\nn\\
&&\leq I\{T_\kappa>S_\kappa\}+|F_\kappa^u(\tilde{T}_\kappa)-F_\kappa^u(\min\{\zeta_\kappa,\tilde{T}_\kappa\})|
\ea
converges to zero in probability by Assumption \ref{sclb} and (\ref{ttctz}).
\end{proof}
\end{lem}
\begin{prop}\label{mrgds}
For $u\in\mathcal{T}$,
\ba
\lim_{\kappa\rightarrow\infty}F_\kappa^u(T_\kappa)=F_\infty^u(\infty)
\ea
in probability.
\end{prop}
\begin{proof}Combine Lemmas \ref{expas} and \ref{gdrisn}.
\end{proof}
Proposition \ref{mrgds} is at the heart of Theorem \ref{mainresult}'s proof. It says that the generalised logistic and exponential models see the same limiting descendant fractions. The descendant fractions view a particular aspect of the population's genealogical structure. The next section translates our viewpoint from the population level to the sample level.

\section{Sampling}\label{samplingsection}
We will sample from the population in a way which makes use of the fact that the descendant fractions converge in probability. Our sampling method will ultimately ensure that the coalescent tree converges in probability too. The sampling structure is defined via subintervals of $[0,1)$.

Let
$$\Theta_\kappa^\emptyset(t)=\left[0,F_\kappa^\emptyset(t)\right)\subset[0,1)
$$
and then recursively for $u\in\mathcal{T}$,
$$
\Theta_\kappa^{uj}(t)=\left[\inf\Theta_\kappa^u(t)+\sum_{i=1}^{j-1}F_\kappa^{ui}(t),\quad\inf\Theta_\kappa^u(t)+\sum_{i=1}^{j}F_\kappa^{ui}(t)\right)\subset[0,1),
$$
for $j\in\mathbb{N}$, $\kappa\in\mathbb{N}\cup\{\infty\}$, and $t\in[0,\infty]$.

The $\Theta_\kappa^u(t)$ have some important properties. First, $|\Theta_\kappa^u(t)|=F^u_{\kappa}(t)$. Second, $u\preceq v$ implies $\Theta_\kappa^u(t)\supset\Theta_\kappa^v(t)$. Third, on the event  $\{N_\kappa(t)>0\}$ and for $s\leq t$, $\left\{\Theta_\kappa^u(t):u\in\mathcal{N}_\kappa(s)\right\}$ is a partition of $[0,1)$. In other words, the $\Theta_\kappa^u(t)$ provide a `stick-breaking' view of the ancestral structure of all individuals alive at time $t$: the size of the interval corresponding to $u$ matches the fraction of living individuals at time $t$ which descended from $u$, and this interval is partitioned by the intervals of $u$'s children (unless $u$ herself is alive at time $t$).

Let $X^1,..,X^m$ be i.i.d. uniform random variables on $[0,1)$ independently of everything else. On the event $\{N_\kappa(T_\kappa)>0\}$, let $V^i_\kappa$ be the unique element of $\mathcal{N}_\kappa(T_\kappa)$ such that $X^i\in \Theta_\kappa^{V^i_\kappa}(T_\kappa)$. On the event of population extinction $\{N_\kappa(T_\kappa)=0\}$, let $V^i_\kappa=\Delta$, where $\Delta\not\in\mathcal{T}$ is a ghost element introduced for technical convenience.

According to the definition above $(V^i_\kappa)_{i=1}^m$ is, on the event $$E_\kappa=\{N_\kappa(T_\kappa)>0\},$$ a uniform sample from $\mathcal{N}_\kappa(T_\kappa)$ \textit{with} replacement. Contrast this with the statement of Theorem \ref{mainresult} where $(U_\kappa^i)_{i=1}^m$ is, on the event $\{N_\kappa(T_\kappa)\geq m\}$, a uniform sample from $\mathcal{N}_\kappa(T_\kappa)$ \textit{without} replacement. Note the connection: $(V_\kappa^i)_{i=1}^m$ conditioned on the event 
$$\tilde{E}_\kappa=\{N_\kappa(T_\kappa)\geq m\text{ and }V^i_\kappa\not=V^j_\kappa,\forall i\not=j\}$$
is distributed as $(U_\kappa^i)_{i=1}^m$.  The following result says that sampling with or without replacement are asymptotically equivalent to conditioning on
$$
E_\infty=\{N_\infty(t)>0,\forall t\geq0\}.
$$
\begin{lem}\label{sampwow}$\lim_{\kappa\rightarrow\infty}I(E_\kappa)=\lim_{\kappa\rightarrow\infty}I(\tilde{E}_\kappa)=I(E_\infty)$ in probability.
\end{lem}
\begin{proof}The event $\{F_\kappa^\emptyset(T_\kappa)=1\}$ coincides with the event $E_\kappa$ for $\kappa\in\mathbb{N}\cup\{\infty\}$. Thus by Proposition \ref{mrgds},
\ba
&&\lim_{\kappa\rightarrow\infty}I(E_\kappa)= I(E_\infty)
\ea
in probability. To see convergence of $I(\tilde{E}_\kappa)$ requires more work. A crucial missing ingredient is that the population size at sampling time is large or extinct; this may be intuitively clear but is so far unproven. Let $n\in\mathbb{N}$. Using Markov's inequality 
\bq
&&\mathbb{P}[1\leq N_\kappa(T_\kappa)\leq n]\nn\\
&&=\mathbb{P}\left[\frac{nI\{N_\kappa(T_\kappa)>0\}}{N_\kappa(T_\kappa)}\geq1\right]\nn\\
&&\leq n\mathbb{E}\left[\frac{I\{N_\kappa(T_\kappa)>0\}}{N_\kappa(T_\kappa)}\right]\nn\\
&&\leq n\mathbb{P}[S_\kappa< T_\kappa]+n\mathbb{E}\left[\frac{I\{N_\kappa(\min\{S_\kappa,T_\kappa\})>0\}}{N_\kappa(\min\{S_\kappa,T_\kappa\})}\right],\label{somebound}
\eq
where $S_\kappa$ is the first time at which the growth rate breaks the superlinear lower bound, which was defined in (\ref{slt}). Adapt (\ref{mart}) to see that
$$
\left(\frac{I\{N_\kappa(\min\{t,S_\kappa,T_\kappa\})>0\}}{N_\kappa(\min\{t,S_\kappa,T_\kappa\})}\right)_{t\geq0}
$$
is a supermartingale. Then (\ref{somebound}) is bounded above by
\bq\label{boundthisone}
\mathbb{P}[1\leq N_\kappa(T_\kappa)\leq n]\leq n\mathbb{P}[S_\kappa< T_\kappa]+n\mathbb{E}\left[\frac{I\{N_\kappa(\min\{\zeta_\kappa,S_\kappa,T_\kappa\})>0\}}{N_\kappa(\min\{\zeta_\kappa,S_\kappa,T_\kappa\})}\right].
\eq
Applying Assumption \ref{sclb} and Lemma \ref{coupeq} to (\ref{boundthisone}),
$$
\lim_{\kappa\rightarrow\infty}\mathbb{P}[1\leq N_\kappa(T_\kappa)\leq n]=0.$$
It follows that
\ba
\mathbb{P}[E_\kappa\backslash\tilde{E}_\kappa]&=&\mathbb{P}[1\leq N_\kappa(T_\kappa)<m]\\&&+\sum_{l=m}^\infty\mathbb{P}[N_\kappa(T_\kappa)=l]\left(1-\frac{l(l-1)..(l-m+1)}{l^m}\right)\\&\leq&\mathbb{P}[1\leq N_\kappa(T_\kappa)<n]+\left(1-\frac{n(n-1)..(n-m+1)}{n^m}\right)\\
&\rightarrow&1-\frac{n(n-1)..(n-m+1)}{n^m}
\ea
as $\kappa\rightarrow\infty$, for any integer $n\geq m$. Therefore $\lim_{\kappa\rightarrow\infty}\mathbb{P}[E_\kappa\backslash\tilde{E}_\kappa]=0$ and the proof is done.
\end{proof}
\section{Coalescence}\label{prooffinishsection}
The most recent common ancestor of $V^i_\kappa$ and $V^j_\kappa$ is defined as
\ba
V^{i,j}_\kappa&=&\max\{u\in\mathcal{T}:u\preceq V_\kappa^i\text{ and }u\preceq V_\kappa^j\}.
\ea
for $\kappa\in\mathbb{N}$ and $i\not=j$, where the maximum is with respect to $\prec$ and $\max\emptyset=\Delta$. Equivalently,
$$
V^{i,j}_\kappa=\max\{u\in\mathcal{T}:X^i,X^j\in\Theta_\kappa^u(T_\kappa)\}.
$$
We can extend the definition to $\kappa=\infty$ by setting
$$
V^{i,j}_\infty=\max\{u\in\mathcal{T}:X^i,X^j\in\Theta_\infty^u(\infty)\}.
$$
\begin{lem}\label{mrcaidentity}For $i\not=j$,
$$
\lim_{\kappa\rightarrow\infty}\mathbb{P}\left[V^{i,j}_\kappa=V_\infty^{i,j}\right]=1.$$
\begin{proof}
By Proposition \ref{mrgds}, the boundaries of the interval $\Theta_\kappa^u(T_\kappa)$ converge in probability to the boundaries of $\Theta_\infty^u(\infty)$. So
$$
I\{u\preceq V_\kappa^{i,j}\}=I\{X^i,X^j\in\Theta_\kappa^u(T_\kappa)\}
$$
converges in probability to
$$
I\{u\preceq V_\infty^{i,j}\}=I\{X^i,X^j\in\Theta_\infty^u(\infty)\}.
$$
The result follows.
\end{proof}
\end{lem}
The coalescence time of $V^i_\kappa$ and $V^j_\kappa$ is 
$$
\tau^{i,j}_\kappa=\sup\{t\in[0,T_\kappa]:V^{i,j}_\kappa\in\mathcal{N}_\kappa(t)\}
$$
for $\kappa\in\mathbb{N}\cup\{\infty\}$ and $i\not=j$, where $\sup\emptyset=\infty$.
\begin{lem}\label{almostfinished}For $i\not=j$,
$$\mathbb{P}\left[\tau_\infty^{i,j}<\infty|\forall t,N_\infty(t)>0\right]=1.
$$
\begin{proof}On the event $\{N_\infty(\kappa)>0\}$, let $\tilde{V}_\kappa^i$ and $\tilde{V}_\kappa^j$ be the unique elements of $\mathcal{N}_\infty(\kappa)$ such that $X^i\in\Theta_\infty^{\tilde{V}_\kappa^i}(\kappa)$ and $X^j\in\Theta_\infty^{\tilde{V}_\kappa^j}(\kappa)$ respectively; meanwhile on the event $\{N_\infty(\kappa)=0\}$, let $\tilde{V}_\kappa^i=\tilde{V}_\kappa^j=\Delta$. That is, $\tilde{V}_\kappa^i$ and $\tilde{V}_\kappa^j$ are randomly sampled with replacement from the exponential population at time $\kappa$. Their most recent common ancestor is
\ba\tilde{V}_\kappa^{i,j}&=&\max\{u\in\mathcal{T}:u\preceq \tilde{V}_\kappa^i\text{ and }u\preceq \tilde{V}_\kappa^j\}\\&=&\max\{u\in\mathcal{T}:X^i,X^j\in\Theta_\infty^u(\kappa)\}.\ea By Lemma \ref{expdf} the boundaries of $\Theta_\infty^u(\kappa)$ converge almost surely as $\kappa\rightarrow\infty$ to the boundaries of $\Theta_\infty^u(\infty)$, and therefore $\tilde{V}_\kappa^{i,j}$ converges almost surely to $V_\infty^{i,j}$. It follows that the coalecence time $$\tilde{\tau}^{i,j}_\kappa=\sup\{t\in[0,\kappa]:\tilde{V}_\kappa^{i,j}\in\mathcal{N}_\infty(t)\}$$as $\kappa\rightarrow\infty$ converges almost surely to $\tau_\infty^{i,j}$. Moreover
$$
\lim_{\kappa\rightarrow\infty}\tilde{\tau}^{i,j}_\kappa I\{N_\infty(\kappa)>0,\tilde{V}_\kappa^i\not=\tilde{V}_\kappa^j\}=\tau^{i,j}_\infty I\{\forall t,N_\infty(t)>0\}
$$almost surely. But by Theorem \ref{Jcor}, $\tilde{\tau}_\kappa^{i,j}$ on the event $\{N_\infty(\kappa)>0,\tilde{V}_\kappa^i\not=\tilde{V}_\kappa^j\}$ converges in distribution to a finite random variable. Thus $\tau_\infty^{i,j}$ on the event $\{\forall t,N_\infty(t)>0\}$ is finite with probability one.
\end{proof}
\end{lem}
\begin{lem}\label{coalestime}For $i\not=j$,
$$
\lim_{\kappa\rightarrow\infty}\mathbb{P}[\tau^{i,j}_\kappa=\tau_\infty^{i,j}]=1.
$$
\begin{proof}
The event $\{\tau^{i,j}_\kappa=\tau^{i,j}_\infty\}$ is a superset of
\bq\label{threeevents}
\{\mathcal{N}_\kappa(t)=\mathcal{N}_\infty(t),\forall t\leq \zeta_\kappa\}\cap\{V^{i,j}_\kappa=V^{i,j}_\infty\}\cap\{\tau^{i,j}_\infty\leq \zeta_\kappa\text{ or }N_\infty(\zeta_\kappa)=0\}.
\eq
By Lemmas \ref{coupeq}, \ref{mrcaidentity} and \ref{almostfinished} respectively, each of the events in (\ref{threeevents}) have probabilities which converge to $1$.
\end{proof}
\end{lem}
\begin{proof}[Proof of Theorem \ref{mainresult}] For $\kappa\in\mathbb{N}\cup\{\infty\}$, the coalescent tree is $(\pi_\kappa(t))_{t\geq0}$, where $\pi_\kappa(t)$ is the partition of $\{1,..,m\}$ such that $i$ and $j$ are in the same block if and only if $\tau_{i,j}^\kappa\geq t$.
Lemma \ref{coalestime} implies joint convergence of the coalescence times, which is equivalent to convergence in finite dimensional distributions of the coalescent tree. By Lemma \ref{sampwow} this convergence holds when sampling with or without replacement. The limiting coalescent tree is measurable with respect to the $X^i$ and $\mathcal{N}_\infty(\cdot)$. That is, the limiting coalescent tree is parameterised by the intrinsic reproduction rates.\end{proof}
%
%

\section*{Acknowledgements}
I thank Alex Mcavoy for supportive and inspiring discussions, Michael Nicholson for shrewd comments on a earlier draft of the paper, and an anonymous reviewer for many thoughtful suggestions and corrections. 
\bibliographystyle{unsrt}

\bibliography{References} 
\end{document}